\documentclass[12pt]{amsart}
\usepackage{geometry}                
\usepackage{graphicx}
\usepackage{tikz}

\usepackage{hyperref}

\usepackage{amssymb}
\usepackage[T1]{fontenc}
\usepackage[bitstream-charter]{mathdesign}

\newtheorem{thm}{Theorem}[section]
\newtheorem{lem}[thm]{Lemma}
\newtheorem{prop}[thm]{Proposition}
\newtheorem{defn}[thm]{Definition}
\newtheorem{cor}[thm]{Corollary}

\newtheorem{remark}[thm]{Remark}
\newtheorem{def-prop}[thm]{Definition-Proposition}

\newcommand{\FA}{\mathbb{A}}
\newcommand{\FF}{\mathbb{F}} 
\newcommand{\FN}{\mathbb{N}}
\newcommand{\FZ}{\mathbb{Z}}  
\newcommand{\FC}{\mathbb{C}}  
\newcommand{\FR}{\mathbb{R}}  
\newcommand{\FQ}{\mathbb{Q}}  

\newcommand{\sL}{\mathcal{L}}

\DeclareMathOperator{\Td}{Todd}
\DeclareMathOperator{\GL}{GL}
\DeclareMathOperator{\SL}{SL}
\DeclareMathOperator{\sign}{sign}

\DeclareMathOperator{\tr}{Tr}
\DeclareMathOperator{\Vol}{Vol}
\DeclareMathOperator{\bbe}{\mathbf{e}}

\title{Equidistribution of generalized Dedekind sums and exponential sums}
\author{Byungheup Jun}
\author{Jungyun Lee}

\email{byungheup@gmail.com} \email{lee9311@kias.re.kr}
\address{School of Mathematics,  Korea  Institute for Advanced Study\\
Hoegiro 87, Dongdaemun-gu, Seoul 130-722, Korea}

\date{Jul. 2. 2013}
\begin{document}

\begin{abstract}
For the generalized Dedekind sums $s_{ij}(p,q)$ defined 
in association with the $x^{i}y^{j}$-coefficient of the Todd power series of the lattice cone in $\FR^2$ generated by 
$(1,0)$ and $(p,q)$, we associate an exponential sum.
We obtain this exponential sum using the cocycle property of the Todd series of 2d cones and the
nonsingular cone decomposition along with the continued fraction of $q/p$.
Its Weil  bound is given  for the modulus 
$q$ applying 
the purity theorem of the cohomology of
the related $\bar{\FQ}_\ell$-sheaf due to  Denef and Loeser.
The Weil type bound of Denef and Loeser 
fulfills the Weyl's equidistribution criterion for  $R(i,j)q^{i+j-2} s_{ij}(p,q)$.
As a special case, we recover the equidistribution
result of the classical Dedekind sums multiplied by $12$ not using the modular weight of the Dedekind's $\eta(\tau)$.
\end{abstract}
\maketitle

\tableofcontents

\section{Introduction}

Classical Dedekind sums $s(p,q)$ are defined for relatively prime integers $p,q$ by
$$
s(p,q)= \sum_{k=1}^q \left(\left(\frac{k}{p}\right)\right)\left(\left(\frac{kp}{q}\right)\right)
$$
where $\left(\left(x\right)\right)$ denotes the value of the 1st periodic Bernoulli function at $x$:
$$
\left(\left(x\right)\right)=
\bar{B}_1(x):= \begin{cases}x-[x]  -\frac12 & \text{for $x\not\in \FZ$} \\ 0 & \text{for $x\in \FZ$} .
\end{cases}
$$
This appears important in describing the change of the Dedekind eta function 
$$
\eta(\tau) = e^{2\pi i \tau /24 }\prod_{n=1}^\infty (1- e^{2\pi i n \tau}), \quad
\text{$\tau\in \frak{h}$},
$$ 
under modular transformations.
$\eta(\tau)$ is a 24th root of the modular discriminant 
$$
\Delta(\tau) = (12 \pi )^{12} \eta^{24}(\tau)
$$ up to some constant.

Due to the modularity after 24th power, 
under modular transformation $\tau \mapsto  A\tau$, for 
$A=\begin{pmatrix} a& b\\ c& d \end{pmatrix} \in \SL_2(\FZ)$, 
its logarithm satisfies the following formula due to Dedekind:
\begin{equation}
\log \eta\left( \frac{a \tau + b}{c\tau + d} \right) = \log \eta(\tau) + 
\frac14 \log\left( - (c\tau + d)^2 \right) 
+ \pi i \phi(A)
\end{equation}
where $\phi(-)$ is the Rademacher's $\phi$-function defined as
a function of $\SL_2(\FZ)$
\begin{equation}\label{phi-function}
\phi(A) :=
\begin{cases}
\frac{a+d}{12c} - \sign c \cdot s(a,c) & \text{for $c\ne 0$} \\
\frac{b}{12d} & \text{for $c=0$}
\end{cases}
\end{equation}
valued in $\frac1{12}\FZ$.
Rademacher's $\phi$-function is valued in $\frac1{12}\FZ$
and many interesting properties of Dedekind sums arise in study of 
$\phi$(Rademacher's $\phi$-function appeared in \cite{RG} is 12 times of \eqref{phi-function} thus valued in $\FZ$).

Beside above modular transformation property, Dedekind sums and their generalization appear in some expressions
of the special values of some $L$- or zeta functions at nonpositive integers. Probably the most famous is the theorem
of Meyer on class number formula in \cite{Me}.  
Siegel and Shintani obtained an expression of special values at nonpositive integers(cf. \cite{Siegel}, \cite{Shintani}).
In particular, Shintani's expression lies in the same line of thought.  One has expression of zeta values using some generalization
of Dedekind sums(\cite{Shintani}, \cite{Pom}). Using this, one can have painless proof of the rationality of the special values at nonpositive integers.  

We are interested  in the distribution of the Dedekind sums and their generalization. 
In \cite{J-L1}, we observed that the random distribution of Dedekind sums is closely related to the integrality of the special values of partial zeta functions. Originally,
Rademacher and Grosswald raised a question on the density of the values of 
Dedekind sums in p.28 of  \cite{RG}. 
There is a long list of publications on this direction. 
Hickerson proved that the points $\left(p/q, s(p,q)\right)$ are dense in $\FR^2$(\cite{Hi}).
Much later, big progress was made independently by Myerson and Vardi.
Vardi showed that for any positive real number $r$, the set
$\left\{\left<rs(p,q)\right>| 0<p<q, (p,q)=1\right\}$ is equidistributied on the interval $[0,1)$
by relating  Kloosterman sums to Dedekind sums(\cite{Vardi}).
In \cite{Myerson}, Myerson applied similar method to show that for a given nonzero real $r$
the graph of the function $p/q\mapsto rs(p,q)$ (modulo 1)
is equidistributed in the unit square.

Let us recall briefly the idea of Vardi and Myerson. 
As $\phi$ is valued in $\frac{1}{12}\FZ$, for the case $r\in 12\FZ$, one can relate the Kloosterman sums to the Dedekind sums  multiplied by $12$(See Thm.\ref{Rademacher} of this article). 
There is a well-known bound of Kloosterman sums for varying modulus due to Weil(\cite{Weil}),
which is crucial step in showing the Weyl's equidistribution criterion for the 
fractional part of Dedekind sums. 
For $r\not\in 12\FZ$, a generalization of Kloosterman sums due to Selberg are associated
to the multiplier system arising from Dedekind's $\eta$-function(\cite{Selberg}).
The strong equidistribution is beyond our interest and our dicussion focuses more on the occurrence of the factor $\frac1{12}$
and its generalization.

This will be reviewed more precisely in Sec.\ref{equid_classical} of this article for our own purpose.

In this article, we note that  many interesting properties of Dedekind sums are consequence of cocycle
condition.
As a cochain over $({\rm P})\SL_2(\FZ)$, the coboundary of the Rademacher's $\phi$-fucntion 
is identified with either the area  or 
the signature $2$-cocycle of torus fibration over a pair of pants(cf. \cite{at1}, \cite{hir}, \cite{K-M}, \cite{Meyer}).
Since these $2$-cocycles have simple geometric
interpretations, one obtains identities involving Dedekind sums
(cf. \cite{asai}, \cite{at1}, \cite{hir}, \cite{K-M}, \cite{Sczech2}).
It is a folklore that all known properties are consequences of the cocycle condition of this sort(e.g. \cite{St}, \cite{Solomon}, 
\cite{Sczech2}).
To name one of them,
one obtains the celebrated reciprocity law for the Dedekind sums by swapping
the two arguments. Also from the (finite length) continued fraction of a rational number  
one can describe it using the terms of the continued fraction(cf. \cite{barkan}, \cite{Hi}, \cite{Knu}). 
Of course, a brute-force computation using the explicit expression recovers the reciprocity law
again, but the explicit form itself is already the consequence of the cocycle condition. 
It is worth to note that Fukuhara identified Dedekind symbols in terms of the reciprocity function(\cite{Fukuhara1}, \cite{Fukuhara2}). From the point of view of cocycle property, after small corrections Dedekind symbols are part of 
1-cocyle and their reciprocity is the coboundary operation followed by specialization at a cusp. The cohomological
understanding is given by Manin through \cite{Manin}, \cite{Manin2} and \cite{Manin-Marcolli} for more general scope
using modular symbols.

The goal of this article is to relate an exponential sum
to a version of generalization of Dedekind sums appearing in 
a paper of Apostol(\cite{apostol})
using the cocycle condition. 
In this setting, the Kloosterman sum appears related to the classical Dedekind sums. 

For $i, j \ge1$, we consider the following 
generalization of Dedekind sums: 
$$
s_{ij}(p,q) : = \sum_{k=0}^{q-1} \bar{B}_i\left(\frac{k}{q}\right) \bar{B}_j\left(\frac{pk}{q}\right) .
$$ 
where $\bar{B}_i(x)$ denotes the periodic Bernoulli function. 
Classical Dedekind sums occur for the case $i=j=1$.
These appeared first in {\it loc.cit.} and Carlitz wrote some papers on their properties(\cite{ca1}, \cite{ca2}).
These vanish for $i+j$ odd after similar reasoning for the vanishing of  Bernoulli numbers of odd degree(Cor.\ref{odd_vanishing}).

There are many ways of writing the Dedekind sums. 
In this paper, so as to treat the generalized ones as well as the classical, we recover
the Dedekind sums as the coefficients of (the germ of) a certain analytic function two variables at 0 in $\FC^2$ associated
to a 2-dimensional lattice cone in $\FR^2$. 
Similar construction was made by Solomon(\cite{Solomon}) on a different basis. 
Garoufalidis and Pommersheim in \cite{Pom} took the same generating function as ours in defining
the generalized Dedekind sums but differ by some power of $q$. 
This point will be clarified later in Sec.\ref{General_Dedekind}.

A lattice cone can be identified with an affine linear map $\sigma: [0,1] \to \FR^2$
such that $\sigma(0),\sigma(1)$ are two linearly independent primitive lattice vectors.
Seen as a singular chain of $\FR^2-0$, we have obvious notion of the boundary operation. 
By 1-cocycle, we mean a functional
$S$ defined on 2-dimensional cones, which vanishes on the boundary of a (degenerate) 3-dimensional cone. This is equivalent to say that 
$$
S(\sigma_1) + S(\sigma_2) = S(\sigma)
$$
for 
$
\sigma_1+\sigma_2 = \sigma
$ 
where the addition of cones is defined as their concatenation.

For a pair of relatively prime positive integers $p, q$(Suppose $ p< q$ for convenience), 
one associates a two variable power series denoted by $\Td_{pq}(x,y)$ defined by Brion-Vergne(\cite{brion}) to write the Euler-Maclaurin summation formular for higher dimensional lattice polytopes(For precise definition, we refer the reader to \SS \ref{Todd_series}). 
$\Td_{pq}(x,y)$ is the Todd series of the cone $\sigma((1,0),(p,q))$, whose coefficient $\frac{t_{ij}(p,q)}{i!j!}$ of $x^{i} y^{j}$ 
 is closely related to $s_{ij}(p,q)$.

In particular, for  nonsingular cones, it is
\begin{equation}\label{nonsingular_todd}
\Td(x,y) = \frac{x}{1-e^{-x}} \frac{y}{1-e^{-y}} = \sum_{i,j} \frac{B_i B_j}{i!j!} x^i y^j
\end{equation}

Todd series, after certain normalization, makes a 1-cocycle as functional over chains of 2-dimensional lattice cones, 
which we call \textit{Todd cocycle} in this article.
Using the cocycle condition, we obtain an explicit formula of Todd series of a cone w.r.t. 
the cone decomposition attached to the continued fraction. Since the cones appearing
in continued fraction are nonsingular, the Todd series is decomposed into the Todd series
of each nonsingular cones.  
As seen in \eqref{nonsingular_todd}, we finally obtain an expression of generalized Dedekind sums 
involving only finite number of Bernoulli numbers. 
Then multiplied by the denominator which appear to be $R(i,j)q^{N-2}$ which would be $12$ for $i=j=1$, 
after the following theorem
we associate certain exponential sums, which generalizes the Kloosterman sums. 
\begin{thm}[main theorem]\label{general}
For $i+j=N\geq 2$ even, we have 
$$
R(i,j)q^{N-2}s_{ij}(p,q)  - \frac{p'^{i}\alpha_{N}r_N\begin{pmatrix}N-1 \\i\end{pmatrix} +p^{j}\alpha_{N}r_N\begin{pmatrix}N-1 \\j\end{pmatrix}}{q}
$$  is always integer 
where $p'$ is an integer such that $p'p \equiv 1 \pmod{q}$ and 
$$
R(i,j):=\begin{pmatrix}
        N \\
       i
\end{pmatrix}  \beta_{N} r_N,
$$
for an integer $r_N$(for precise description, see Thm.\ref{Todd_N_part})
and $\beta_N$ being the denominator of $B_N$ the $N$-th Bernoulli number.
\end{thm}

Thm.\ref{general} 
relates a particular case of the generalized Kloosterman sum $K_{ij}(k,\ell,q)$ for
$k=\alpha_{N}r_N {N-1 \choose i}$ and $\ell=\alpha_{N} r_N {N-1 \choose j}$
defined below:
\begin{defn}[Generalized Kloosterman sum]\label{G_KS}
For a positive integer $q$,
$$
K_{ij}(k,\ell,q):=
\sum_{\substack{0<p<q\\pp'\equiv 1\pmod{q}}}\bbe\left(\frac{k(p')^{i}+\ell p^{j}}{q}\right),$$
where $\bbe(x):=\exp(2\pi i x)$.
\end{defn}

These generalized Kloosterman sums have Weil type bound. It is a consequence of 
a result on the weight and dimension of the cohomology of an $\ell$-adic sheaf 
due  to Denef-Loeser in \cite{De-Loe}.
Using the Weil bound for the generalized Kloosterman sums, one can show that
the Weyl's equidistribution criterion for the generalized Dedekind sums holds.

Thus after our main theorem, the equidistribution property of the (fractional part) of the generalized Dedekind sums multiplied by
$R(i,j)q^{N}$ similar to 
that on classical Dedekind sums in \cite{Vardi}.
\begin{thm}
\label{main}
 For even $N=i+j$,
the set 
\begin{equation*}
\left\{
\left<R(i,j)q^{N-2}s_{i,j}(p,q)\right>\big|0<p<q,(p,q)=1
\right\}
\end{equation*}
is equidistributed in the interval $[0,1)$, where $\left<x\right>$ is the  fractional part of $x$ in $[0,1)$.
\end{thm}

Said roughly, the equidistribution of the classical Dedekind sums multiplied by $12$ in \textit{loc.cit.} is a consequence of the modularity of $\eta(\tau)$.
Contrary to the classical case, 
as we don't have such a function which play the role of $\log\eta(\tau)$ for generalized Dedekind sums, 
the theorem is not entirely clear from the definition. 
Thus we argue that this is a consequence the cocycle property of the Todd series similar to many other properties 
of the classical Dedekind sums. As mentioned already, for classical Dedekind sums we have  $12=R(1,1)q^0$.


This paper is composed as follows: In Sec.2, we review a part of Vardi's result to relate the Kloosterman sum to 
the classical Dedekind sums. In Sec.3, we define the Todd series of a lattice cone and describe the cocycle 
condition. In Sec.4, the generalized Dedekind sums are identified with the coefficients of the Todd series.
In Sec.5 the generalized Kloosterman sums appear in relation to the generalized Dedekind sums
and the proof of Thm.\ref{general} is given.
Sec.6 is devoted to the Weil bound for generalized Kloosterman sums and we finish the proof of the main 
theorem.

\begin{center}
\bf Notations and convention
\end{center}

\begin{itemize}
\item For a real number $x$, $\left< x\right>= x - [x]$ is the fractional part taken in $[0,1)$.
\item For a function $f$ in $x$, $ f << x^{a+\epsilon}(\forall \epsilon>0)$ means
that $|f| = o(x^{a+\epsilon})$ for every positive $\epsilon$.
\item $\bbe(x)$ denotes $\exp(2\pi i x)$.
\item The $k$-th \textit{Bernoulli number} $B_i$ is defined by the generating function
$$
\frac{z}{e^z-1} = \sum_{k=0}^\infty \frac{B_k}{k!} z^k.
$$
\item The $k$-th \textit{Bernoulli polynomial} $B_k(x)$ is the degree $k$ polynomial  defined by
$$
\frac{ze^{xz}}{e^z-1} = \sum_{k=0}^\infty \frac{B_k(x)}{k!} z^k.
$$
The \textit{$k$-th periodic Bernoulli function} $\bar{B}_k(x)$ is a $\FZ$-periodic function on $\FR$  
defined by assigning the values for $x\in [0,1)$ as follows:
$$
\bar{B}_k (x) = 
\begin{cases}
B_k (\left<x\right>) & \text{for $x\not\in \FZ$}\\
B_k & \text{for $k\ge 2$ and $x\in\FZ$,} \\
0 & \text{for $k=1$, $x\in \FZ$}
\end{cases}
$$
\end{itemize}

\textbf{Acknowledgement.} 
We thank  Hi-joon Chae and  Haesang Sun for careful reading of earlier version of this article. 
We thank Prof. Y. I. Manin for showing interest and for introducing works of S. Fukuhara. 
We are also grateful to Vincent Maillot for encouragement, many valuable comments and discussion.

\section{Equidistribution of classical Dedekind sums}\label{equid_classical}

In this section, we sketch the proof of the equidistribution of the fractional parts of  classical Dedekind sums
multiplied by 12 using the Weil's bound for Kloosterman sums. 
The proof appears at the beginning of Vardi(\cite{Vardi}). This is not only a special case of the main result of \textit{loc.cit.}
but also a case not covered by the main technic of multiplier system attached to $\log \eta(\tau)$.

As discussed, this step comes from the modularity of $\eta(\tau)$ in relation to 
$\Delta(\tau)$.

We remind a criterion for a sequence to be equidistributed in $[0,1)$ due to H. Weyl and will show that
this is the case.
Later, we will be using the  Weyl's equidistribution criterion for the generalized Dedekind sums.

\subsection{Rademacher's theorem}
We begin with a reinterpretation of Rademacher's $\phi$-function. 
\begin{thm}[Rademacher]\label{Rademacher}
For a relatively prime pair of integers $(p,q)$, 
$12 s(p,q) -\frac{p'+p}{q}$ is always integer whenever $p'p\equiv 1 \pmod {q}$.
\end{thm}

The above Rademacher's theorem is nothing but rephrasing the fact that the values of Rademacher's $\phi$-function
are taken in $\frac{1}{12}\FZ$.
At a glance, this is not along the line we follow to show the result of same type for generalized Dedekind sums. 
But later we will see that the integrality will turn out to be a special case of the cocycle property. 
Later in Sec.5, we will see that the appearance of $12$ in the Rademacher's theorem is due to the denominators are 
product of  $B_1$ and $B_2$ the 1st and  the 2nd Bernoulli numbers.

\subsection{Kloosterman sums}
Let $q$ be a positive integer and $k, \ell$ be a pair of integers relatively prime to $q$.
The \textit{Kloosterman sum} for $k,\ell$ of modulus $q$ is denoted by
$K(k,\ell,q)$ and defined as
\begin{equation*}
K(k,\ell,q): = \sum_{x \in (\FZ/q\FZ)*} \bbe \left(\frac{k}q x + \frac{\ell}q x^{-1}\right)
 = \sum_{\substack{0\le x \le q-1\\ (x,q)=1}} \bbe  \left(\frac{k}q x + \frac{\ell}q x^{-1}\right)
\end{equation*}

There is a well-known upper bound  of Kloosterman sums due to Weil: 
\begin{equation}\label{kloo}
K(k,\ell,q)
<< q^{\frac{1}{2}+\epsilon}\quad(\forall \epsilon >0).
\end{equation}

\subsection{Kloosterman sum and Dedekind sum}\label{KDS}
For two positive integers $m$ and $q$ fixed, if we sum $\bbe\left(12 m s\left(p,q\right)\right)$ 
over $1 \le p \le q$ such that $(p,q)=1$, 
we obtain the Kloosterman sum from Thm.\ref{Rademacher}:
\begin{equation}\label{dedekind}
K(m,m,q) = 
\sum_{\substack{0<p<q\\(p,q)=1}} \bbe\left(12m s(p,q)\right)
= \sum_{\substack{0<p<q\\p'p \equiv 1 \pmod{q}}}\bbe\left(\frac{p'm+pm}{q}\right).
\end{equation}

Using Weil's bound (\ref{kloo}),   we conclude 
that
\begin{equation}\label{estimation}
\sum_{0<q<x}
\sum_{\substack{0<p<q\\(p,q)=1}} 
\bbe\left(12 m s(p,q)\right)<< x^{\frac{3}{2}+\epsilon},\,\,\,\,(\forall \epsilon>0).
\end{equation}

Then the above bound implies that
the set 
$$
\left\{\left<12 s (p, q)\right>\Big| 0<p<q, (q,p)=1
\right\}
$$
fulfills a famous criterion for a sequence in $[0,1)$ to be equidistributed due to H. Weyl(\cite{Weyl}) 
as stated below.

\subsection{Weyl's equidistribution criterion}
A sequence $\{s_i \in [0,1) \}_{i\in \FN}$ is equidistributed 
iff for every  $k\in \FZ\backslash\{0 \}$, 
$$
\lim_{N\to \infty} \frac1N \sum_{n=1}^N \bbe( k s_n)= 0.
$$ 

For a positive integer $m$ and $x$, let $E(m,x)$ be the value
$$
E(m,x) = \frac{1}{\# \left\{(p,q)|\gcd(p,q)=1, p<q\le x \right\} } \sum_{0<q<x}\sum_{\substack{0<p<q\\(p,q)=1}} \bbe\left(12 m s\left(p,q\right)\right).
$$ 
Taking the limit $x\to \infty$, from the Weil's bound, we have $E(m,x)\to 0$. 
Therefore $\left<12 s(p,q)\right>$ is equidistributed in $[0,1)$.


For the rest of the paper, 
we consider the possibility of the same sort of  equidistribution of the 
generalized Dedekind sums $s_{ij}(p,q)$.

\section{Todd Series of  a lattice cone}
Now we recall the definition of Todd series as defined by Brion-Vergne in \cite{brion}, for the case of 2-dimensional cones.
For 1 dimensional case, this equals the generating function of the Bernoulli numbers up to 
sign of the variable. The Todd series yields a differential operator of infinite order used in the formulation
of the Euler-Maclaurin formula for higher dimensional polytopes. 
These generating functions are  closely related to Shintani functions considered by Solomon in \cite{Solomon}
and have similar cocycle property. Later we will see that the Todd series recovers the generalized Dedekind sums
as Solomon's Shintani functions. But we warn the reader that these are not the same, 
\textit{a priori} their
cocycle conditions are on the dual cones to each other. This is briefly mentioned by Garoufalidis-Pommersheim in \cite{Pom}.
We will clarify this relation in another paper in sequel \cite{J-L2}. 
Maybe we could conclude the same result in the framework of Shintani functions. But through our works in relation(\cite{J-L1},
\cite{J-L2}), 
we find it more comfortable to manipulate Todd series rather than Solomon's Shintani functions. 

\subsection{Lattice cones}
Let $M$ be the standard lattice $\FZ^2$ in $\FR^2$. We consider cones defined in $M$.
By lattice cone, we mean the convex hull of two linearly independent rays of rational slopes. 
It is always possible to choose unique primitive lattice vectors generating the rays. 
Let $\sigma = \sigma(v_1,v_2)$ be a lattice cone and $v_1,v_2$ be primitive lattice generators of the 
rays bounding $\sigma$.
Be aware that we take the orientation(ie. the order of the rays) into consideration so that $\sigma(v_1,v_2) \neq \sigma(v_2,v_1)$. 
$\sigma$ is sometimes identified with an integer coefficient matrix 
$A_\sigma$ whose columns are the lattice vectors $v_1,v_2$ in $\FZ^2$.
$M_\sigma$ denotes the sublattice of $M$ generated by $v_1,v_2$. 
$\Gamma_\sigma = M/M_\sigma$ is isomorphic to a cyclic group of order $\left|\det(A_\sigma)\right|$.

For $g \in M$ representing $\gamma\in \Gamma_\sigma$,  we have rational numbers $a_{\sigma,i}(g)$, $i=1,2$ such that
$$
g = a_{\sigma,1}(\gamma)v_1 + a_{\sigma,2}(\gamma) v_2 .
$$

$a_{i,\sigma}$, being integral on $M_\sigma$, yields a character  $\chi_{\sigma,i}$ on $\Gamma_\sigma$
as  
$$
\chi_{\sigma,i} (\gamma) := \bbe\left(a_{\sigma, i}(g)\right),\quad \text{for $i=1,2$}.
$$

\subsection{Todd Series}\label{Todd_series}
The Todd power series of $\sigma$ is defined as:
\begin{equation}\label{todd_definition}
\Td_\sigma (x_1,x_2): = \sum_{\gamma\in \Gamma_\sigma} \frac{x_1}{1-\chi_{\sigma,1}(\gamma) e^{-x_1}}
\frac{x_2}{1-\chi_{\sigma,2}(\gamma) e^{-x_2}}
\end{equation}

The coefficients of $\Td_\sigma(x_1,x_2)$ is rational though the expression \eqref{todd_definition} 
contains some roots of $1$. This is easy to see from the Galois invariance of the expression.

The Todd series is invariant of  the $\SL_2(\FZ)$ equivalent class of  cones by the following proposition.

\begin{prop}
Let $\sigma = \sigma(v_1,v_2)$ be a be a lattice cone 
and $A$ be a matrix in $\SL_2(\FZ)$.
Then we have
$$
\Td_\sigma (x_1,x_2) = \Td_{A\sigma} (x_1,x_2).
$$
\end{prop}
\begin{proof}
$A$ gives an isomorphism
$$
A: M/M_\sigma \to M/M_{A\sigma},\quad\gamma\mapsto A\gamma,\
\text{for $\gamma\in \Gamma_\sigma$}.
$$
For $\gamma = a_{\sigma,1}(\gamma) v_1 + a_{\sigma,2} (\gamma) v_2$,
$$
A\gamma = a_{\sigma,1} (\gamma) A v_1 + a_{\sigma,2}(\gamma) A v_2.
$$
Therefore
\begin{equation*}
\begin{split}
\Td_{A\sigma}(x_1,x_2) &= 
	\sum_{A\gamma\in \Gamma_{A\sigma}} 
		\frac{x_1}{1-\chi_{A\sigma,1}(A\gamma)e^{-x_1}}
		\frac{x_2}{1-\chi_{A\sigma,2}(A\gamma)e^{-x_2}} \\
	& = \sum_{A\gamma\in \Gamma_{A\sigma}} 
		\frac{x_1}{1-\chi_{\sigma,1}(\gamma)e^{-x_1}}
		\frac{x_2}{1-\chi_{\sigma,2}(\gamma)e^{-x_2}} \\
	&=\Td_\sigma(x_1,x_2)
\end{split}
\end{equation*}
\end{proof}

Let $p, q>0$ be two relatively prime nonnegative integers. Then $(1,0)$ and $(p,q)$ are
primitive lattice vectors and linearly independent. 
Let $\sigma_{pq}$ denote the cone generated by
$(1,0)$ and $(p,q)$.   
Notice that any lattice cone is equivalent to $\sigma_{pq}$ after basis change.
We shall write $\Td_{pq}$ instead of $\Td_{\sigma_{pq}}$ for simplicity.

\subsection{Normalized Todd series and cocycle property}

From now on, we will be dealing with only lattice cones in the 1st quadrant. 
This is not necessary in defining the cocycle property for cones but otherwise we need to
extend the category of cones due to the unnecessary occurrence of the Maslov index(cf. \cite{arnold}, \cite{J-L2}, \cite{Solomon}). 
For example, in \cite{Solomon}, this ambiguity appeared by the name  `\textit{formal Cauchy theorem}'.
Consequently, for full generality, one has to take the value of the cocycle modulo $\FZ$ or to take a 
central extension of $\SL_2(\FZ)$(cf.\cite{at1}, \cite{K-M}).
Here considering  only the cones in the 1st quadrant, we can avoid this difficulty.
A drawback is that we don't have the cocycles defined over $\GL_2(\FQ)$ but over singular chains in $\FR^2-0$.
Nevertheless, this won't harm any result of this article. 

\begin{defn}
Let $\sigma$ be a lattice cone.
Then the  normalized Todd series $S_\sigma(x_1,x_2)$ of $\sigma$ is defined as
$$
S_{\sigma} (x_1,x_2) = \frac{1}{\det(A_\sigma) x_1 x_2 } \Td_\sigma(x_1,x_2).
$$
Similarly, $S_{\sigma_{pq}}$ is abbreviated to $S_{pq}$ as in unnormalized case.
\end{defn}

Because $\Td_\sigma(x_1,x_2)$ is holomorphic at $0\in \FC^2$, we may well take
$\Td_\sigma(x_1,x_2)\in \FC\{\{ x_1, x_2\}\}$, where $\FC\{\{ x_1,x_2\}\}$ is the
ring of  power series convergent for some neightborhood of $0$. 
Note that the coefficients of $\Td_\sigma(x_1,x_2)$ lie in $\FQ$.
Thus we can write
$$
\Td_\sigma(x_1,x_2) \in \FQ\{\{ x_1,x_2\}\}:= \FC\{\{ x_1, x_2\}\}\cap\FQ[[x_1,x_2]].
$$

With the notations above,  since $S_\sigma(x_1,x_2)$ has 
simple pole along the two axes: $x_1 =0$ and $x_2=0$,
$$S_\sigma(x_1,x_2) \in \frac1{x_1x_2}\FQ\{\{x_1,x_2\}\}.$$

Note that swapping  two rays of the cone interchanges not only the variables but also the 
sign in $S_\sigma$. Thus the orientation of a cone is reflected in $S_\sigma$.
In this case, the same cone with the opposite orientation will be denoted by $-\sigma$. 

\begin{lem}
For a lattice cone $\sigma$ in the 1st quadrant, 
\begin{align*}
S_{-\sigma} (x_1,x_2) &= - S_{\sigma}(x_2,x_1) \\
\Td_{-\sigma} (x_1,x_2) & = \Td_\sigma(x_2,x_1)
\end{align*}

\end{lem}

Let $v_1,v_2, v_3$ be pairwise linearly independent primitive lattice vectors in the 1st quadrant.
Let $\sigma_{ij}$ be the lattice cone generated by $v_i, v_j$. 
Then we write $\sigma_{ik}=\sigma_{ij}+\sigma_{jk}$.
Actually a cone $\sigma$ generated by lattice vectors  $v_1,v_2$ can be seen as a simplex
$$
\sigma: [0,1] \to \FR^2
$$
for an affine linear map $\sigma$ with $\sigma(0)=v_1, \sigma(1) = v_2$.
Abusing the notation, if $v_1=v_2$, such a degenerate cone behaves as a unit when
added to $\sigma(v_1,v)$ for any $v$. Thus we consider the groupoid of 2 dimensional
cones. The addition is defined up to boundary of (degenerate) 3 dimensional cones.

\begin{defn}
A \textit{1-cocycle} over cones is a functional over cones valued in an abelian group $M$
$$
\phi: \sigma \mapsto \phi(\sigma)\in M
$$
satisfying $\phi(\sigma_1 + \sigma_2) = \phi(\sigma_1) + \phi(\sigma_2)$.
\end{defn}

In terms of groupoids, the 1-cocycle is a map preserving the operation of the groupoid of cones.
Our notion of 1-cocycles agree with the modular pseudo-measures defined by Manin and Marcolli(
\cite{Manin-Marcolli}), if it were defined on $\Bbb{P}^1_+(\Bbb{R}):= \FR^2-0/\FR^*_+$.
Furthermore, the action of $SL_2(\FZ)$ on $\FR^2$ makes the 1-cocycle a modular pseudo-measure.

To avoid unnecessary complication, we will consider only the cones lying in the right half plane 
$\FR^2_{x_1\ge0}$.
From now on, we denote by `$\text{2-d Cones}$' the set of all 2-dimensional lattice 
cones in the right half plane.
Todd cocycle is a cocycle on $\text{2-d Cones}$ defined as below.
Let $L$ be the reduced equation of the line orthogonal to  a lattice vector $(p,q)$ for $p>0$. Written explicitly, 
$$
L =  q x_1 - p x_2.
$$
By $\frac{1}{\prod_L L} \FQ\{\{x_1,x_2\}\}$, we mean 
$$
\sum_L \frac{1}{L}\FQ\{\{x_1,x_2\}\}
$$
where $L$ runs for all primitive lattice vectors in the right half plane.
\begin{def-prop}[Todd cocycle]
The  Todd cocycle is a map
$$
\Phi : \text{2-d Cones} \to 
\frac{1}{\prod_L L} \FQ\{\{x_1,x_2\}\}
$$ 
given by
$$
\Phi(\sigma) = S_{\sigma} (A_\sigma^{-1}(x_1,x_2)) \in \FQ\{\{x_1,x_2\}\} (L_1^{-1},L_2^{-1}),
$$
where $L_i$ are the equation of lines orthogonal to $v_1,v_2$.
This is a 1-cocycle over 2-dimensional cones in $\FR^2$. 
\end{def-prop}

\begin{proof}
For the proof, we refer the reader to Thm.3 of \cite{Pom2}.
\end{proof}

\section{Generalized Dedekind sum as coefficients of  Todd series of a cone}
\label{General_Dedekind}
Now we are going to identify the generalized Dedekind sum $s_{ij}(p,q)$ with the $x_1^i x_2^j$-coefficient of
$\Td_{pq}$.
We begin with the definition of $s_{ij}(p,q)$. 
There are variations of the same sum in essential for instance \cite{Pom}, \cite{Solomon}.
Our convention follows that appeared in p.71 of  \cite{RG}.

Recall that we consider for positive integers $i,j$,   the generalized Dedekind sum as follows:
$$
s_{ij}(p,q):=\sum_{k=0}^{q-1}\bar{B}_{i}(\frac{k}{q})\bar{B}_{j}(\frac{pk}{q}).
$$

For $i=j=1$, we have the classical Dedekind sum: $s_{11}(p,q) = s(p,q)$.

 Let $\frac{t_{ij}(p,q)}{i!j!}$ be the coefficient of 
$x_1^ix_2^j$ in  $\Td_{pq}(x_1,x_2)$. 
Thus 
$$\Td_{pq}(x_1,x_2) = \sum_{i,j \ge 0} \frac{t_{ij}(p,q)}{i! j!} x_1^i x_2^j.$$ 

\begin{thm}[Compare with \cite{Pom}]\label{coefficient}
We have 
$$
t_{i j }(p,q)=-(-q)^{i+j-1}\left(s_{ij}(p,q)+\delta(i,j)B_i B_j\right),
$$
where 
$\delta(i,j)=\begin{cases}1,& i=1 \text{ or }j=1,
\\0, &\text{ otherwise}\end{cases}$ and $B_i$ is the $i$-th
Bernoulli number.
\end{thm}

\begin{proof}
Let $\sigma=\sigma_{pq}$ and $v_1=(1,0)$, $v_2=(p,q)$ 
be the primitive nonzero lattice generators of $\sigma$.
Let $v^*_j$ be the dual vector of $v_i$ for $i=1,2$(ie. $\left<v_i, v_j^*\right> = \delta_{ij}$). 
Thus 
$$
v_1^* = \left(1, -\frac{p}q\right)\quad\text{and}\quad v^*_2  = \left(0,\frac1q \right).
$$

The dual cone $\check\sigma$ of $\sigma$ is generated by $v_1^*$ and $v_2^*$.

Then 
we have
\begin{equation}\label{t1}
\begin{split}
\Td_{pq}(x_1,x_2)&=\sum_{g\in \Gamma_{\sigma_{pq}}}\frac{x_1}{1-e^{2\pi i \left<v_1^*,g\right>}e^{-x_1}}\frac{x_2}{1-e^{2\pi i \left<v_2^*,g\right>}e^{-x_2}}\\
&=x_1x_2\sum_{n_1, n_2\geq0}\sum_{g\in \Gamma_{\sigma_{pq}}}e^{2\pi i \left<n_1v_1^*+n_2v_2^*,g\right>}e^{-n_1x_1-n_2x_2}.
\end{split}
\end{equation}

Since 
$$\sum_{g\in \Gamma_{\sigma_{p,q}}}e^{2\pi i \left<n_1v_1^*+n_2v_2^*,g\right>}=
\begin{cases} |\Gamma_{\sigma_{p,q}}|=q  ,& n_1v_1^*+n_2v_2^*\in M^*,
\\0, &\text{ otherwise}\end{cases}
$$
one can write $\Td_{pq}$ as summation over lattice points in $\check\sigma$:
\begin{equation}\label{t2}
\begin{split}
\Td_{pq}(x_1,x_2)&= 
qx_1x_2
\sum_{\substack{n_1v_1^*+n_2v_2^*\in M^*\\ n_1,n_2 \geq 0}}e^{-n_1x_1-n_2x_2} \\
&=qx_1x_2\sum_{m\in \FZ^2 \cap \check\sigma}e^{-\left<m,v_1\right>x_1-\left<m,v_2\right>x_2}.
\end{split}
\end{equation}

Notice that in general $v_i^*$ are not lattice vectors
but the primitive lattice generators  of $\check\sigma$ 
are 
$$u_1=(q,-p),\,\,u_2=(0,1).$$ 
Let $P(u_1,u_2)$ the following half open parallelogram: 
$$
P(u_1,u_2)=\left\{x_1u_1+x_2u_2 |0\leq x_i<1\right\}.
$$
Then 
$$
\check\sigma\cap M^*=\left\{z+n_1u_1+n_2u_2| z\in P(u_1,u_2)\cap M^*, n_i\geq 0\right\}.
$$
Thus we have
\begin{equation}
\begin{split}
&\Td_{pq}(x_1,x_2)=q x_1 x_2 \sum_{z\in P(u_1,u_2)\cap
M^*}\frac{e^{-\left<z,v_1\right>x_1-\left<z,v_2\right>x_2}}{(1-e^{-qx_1})(1-e^{-q
x_2})}\\
&=q^{-1}\sum_{i,j\geq 0} \sum_{z\in P(u_1,u_2)\cap
M^*}B_i\left(\frac{\left<z,v_1\right>}{q}\right)B_j\left(\frac{\left<z,v_2\right>}{q}\right)(-q)^{i+j} \frac{x_1^i x_2^j}{i!j!}
\end{split}
\end{equation}

The lattice points inside $P(u_1,u_2)$ are identified as follows:
$$
P(u_1,u_2)\cap M^*=\left\{\frac{k}{q}u_1+
\left<\frac{pk}{q}\right>
u_2
\Big| \,\, k=0,1,2,\cdots, q-1\right\}.
$$
Hence $t_{ij}(p,q)$ and $s_{ij}(p,q)$ are related in the desired form:
\begin{equation*}
\begin{split}
t_{ij}(p,q) &= q^{-1}(-q)^{i+j}\sum_{z\in P(u_1,u_2)\cap
M^*}B_i\left(\frac{\left<z,v_1\right>}{q}\right)B_j\left(\frac{\left<z,v_2\right>}{q}\right) \\
&=q^{-1}(-q)^{i+j}\sum_{k=0}^{q-1} B_i\left(\frac{k}{q}\right)B_j\left(\left<\frac{pk}{q}\right>\right)\\
&=q^{-1}(-q)^{i+j}\left(\sum_{k=0}^{q-1} \overline{B}_i\left(\frac{k}{q}\right)\overline{B}_j\left(\frac{pk}{q}\right)+\delta(i,j)B_i B_j\right)\\
&= -(-q)^{i+j-1} \left( s_{ij}(p,q) + \delta(i,j) B_i B_j
\right).
\end{split}
\end{equation*}
 \end{proof}


For odd $i+j$,
$s_{ij}(p,q)$ will turn out to be trivial. 
This is easy consequence of  the previous theorem.

Let us
define $L^\lambda(x)$  for a fixed complex number $\lambda\ne0$ as
$$
L^\lambda(x) := \frac{x}{2}\frac{1+\lambda e^{-x}}{1-\lambda e^{-x}}.
$$
For $\lambda = 1$, this is the even part of $\Td(x)$:
$$
\Td(x) :=\frac{x}{1-e^{-x}}=  \frac{x}{2}+ L^{\lambda=1}(x)
$$

For $\lambda\ne 1$, $L^\lambda(x)$ is not even in general, but we have
\begin{equation}
\Td^\lambda(x) := \frac{x}{1-\lambda e^{-x}} = \frac{x}{2} + L^\lambda (x)
\end{equation}
and
\begin{equation}
L^\lambda(-x) = L^{\lambda^{-1}}(x).
\end{equation}

Thus if  $\sigma$ is a lattice cone, 
$
\sum_{g\in \Gamma_\sigma} L^{\chi_i(g)}(x_i) 
$
is even for $i=1,2$. 
Therefore we have decomposition of $\Td_\sigma(x_1,x_2)$ as follows:
\begin{equation}
\Td_\sigma(x_1,x_2) = \frac{q}{4}x_1 x_2 + 
\frac{1}{2} \sum_{g\in\Gamma_\sigma} \left(x_1 L^{\chi_2 (g) }(x_2) + x_2
 L^{\chi_1 (g) }(x_1)\right) +
\frac14 \sum_{g\in\Gamma_\sigma} L^{\chi_1 (g)} (x_1) L^{\chi_2 (g) }(x_2)
\end{equation}
Notice that the odd part of $\Td_\sigma(x_1,x_2)$ is 
$$
\frac{1}{2} \sum_{g\in\Gamma_\sigma} \left(x_1 L^{\chi_2 (g) }(x_2) + x_2
 L^{\chi_1 (g) }(x_1)\right).
$$
So $t_{ij}(p,q)=0$ for $i+j$ odd and $i,j>1$.
Otherwise, for example $i=1$ and $j=2k$,
$$
t_{1,2k}(p,q) = \frac{q^{2k}}2 B_{2k} = -q^{2k}\left( s_{1,2k}(p,q) +B_1 B_{2k}\right).
$$
Since $B_1 =-\frac12$, again we have $s_{1,2k}(p,q) =0$.

Hence we obtain the following corollary:
\begin{cor}\label{odd_vanishing}
Let $p,q$ be relatively prime pair of integers. Then
for given $i,j \ge 1$ such that $i+j$ is odd, 
$$
s_{ij}(p,q) = 0.
$$
\end{cor}

For the rest of this paper, we assume $i+j$ is even.

\section{ Generalized Dedekind sums and Generalized Kloosterman sums}

In this section, we would like  to evaluate  $s_{ij}(p,q)$ 
in terms of the 
the continued fraction of $q/p$ for even
$i+j=N$. 
As we saw in the previous section, 
$s_{ij}(p,q)$ vanishes if $i+j$ is odd.

Writing explicitly $s_{ij}(p,q)$, we will obtain an analogous statement to Rademacher's 
theorem(Thm.\ref{Rademacher}).
Then we will be able to relate generalized Kloosterman sums to a generalization of Dedekind sums.

The geometric counterpart of the continued fraction is the
cone decomposition. 
Accordingly, 
the (normalized) Todd series is decomposed into sum of the (normalized) Todd series
of nonsingular cones.

Let $q$ and $p$ be relatively prime positive integers  and suppose $q>p$.

We are going to associate the (positive) continued fraction of  $q/p$: 
$$
\frac{q}{p}=a_1+ \cfrac{1}{a_2 + \cdots\cfrac{1}{a_n}},$$
where $a_i\geq1$ are all integers. 
We put $(p_{-1},q_{-1})=(1,0)$,
$(p_0,q_0)=(0,1)$ and for $i\geq1$. 
Define a pair of relatively prime integers $p_i$ and $q_i$ using truncation of the continued fraction of $\frac{q}{p}$:
$$
\frac{q_i}{p_i}:=a_1+ \cfrac{1}{a_2 + \cdots\cfrac{1}{a_i}}.
$$

\begin{figure}
\begin{tikzpicture}[scale=2]
\draw [<->] (0,3) -- (0,0) -- (3,0);
\draw [fill] (1,0) circle [radius=0.5pt];
\node [below] at (1,0) {$v_{-1}$};
\draw [fill] (0,1) circle [radius=0.5pt];
\node [left] at (0,1) {$v_0$};
\draw [fill] (2,2.5) circle [radius=0.5pt];
\draw [-] (0,0) -- (2.2, 2.75);
\draw [fill] (1.5, 1.3) circle [radius=0.5pt]; \node [right] at (1.5,1.3) {$v_1 = (p_1,q_1)$};
\draw [-] (0,0) -- (2.5, 2.16);
\draw [fill] (1.2, 1.9) circle [radius=0.5pt]; \node [above left] at (1.2, 1.9) {$v_2=(p_2,q_2)$};
\draw [-] (0,0) -- (1.8,2.85);
\node [below right] at (2,2.5) {$v_n=(p,q)$};
\end{tikzpicture}
\caption{$\nu_i$}\label{figure1}
\end{figure}

As previous, 
let $\sigma:=\sigma_{p,q}$ and $v_k$ be the primitive lattice vectors in the 1st quadrant $(p_k,q_k)$ for   $-1\leq k \leq n$(See Fig.\ref{figure1}).

Then we have the following virtual cone decomposition of $\sigma$
into nonsingular cones: 
$$
\sigma:= \sigma_{pq}= \sigma(v_{-1},v_{n})=\sum_{k=-1}^{n-1} \sigma_k,
$$
where $\sigma_k := \sigma(v_k, v_{k+1})$.

After the additivity of normalized Todd series, according to the continued fraction of $q/p$, we obtain the following expression:
\begin{equation}\label{normalized_Todd}
S_{pq}(x,y)=\sum_{k=-1}^{n-1}(-1)^{k+1}F\left(A_{\sigma_k}^{-1}A_{\sigma}(x,y)^{t}\right),
\end{equation}
where
$$F(x,y)=\frac{1}{1-e^{-x}}\frac{1}{1-e^{-y}}=\frac{\Td (x,y)}{xy}.$$
One should note that $F(x,y)$ is the normalized Todd series of a nonsingular cone (up to sign).

Recall
that the 1-variable Todd series is 
$$
\Td(z) = \frac{z}{1-e^{-z}}=\sum_{i=0}^{\infty}(-1)^i\frac{B_i}{i!} z^{i}.
$$

The matrix $A_{\sigma_k}^{-1} A_\sigma$ is computed as
$$
A_{\sigma_k}^{-1} A_\sigma = (-1)^{k+1} 
\begin{pmatrix}
q_{k+1} & p q_{k+1} - q p_{k+1} \\
-q_k & - pq_{k} + q p_k
\end{pmatrix}.
$$

As $\det(A_\sigma)= q$, by multiplying $qxy$  we obtain the following expression of  Todd series of $\sigma$ from 
\eqref{normalized_Todd}:
\begin{equation}\label{S_sigma}
\begin{split}
\Td_{pq}(x,y)&=qxy\sum_{k=-1}^{n-1}(-1)^{k+1}\frac{\Td(M_k)\Td(M_{k+1})}{M_k M_{k+1}}\\
&=qxy\sum_{k=-1}^{n-1}(-1)^{k+1}\sum_{i=0}^{\infty}\sum_{j=0}^{\infty}(-1)^{i+j}\frac{B_i}{i!}\frac{B_j}{j!}M_k^{j-1}M_{k+1}^{i-1},
\end{split}
\end{equation}
where 
\begin{equation}
M_k := 
\begin{cases}
qy, &   k=-1\\
(-1)^{k}\left(q_kx+(p q_k-q p_k)y\right), & 0\le k \le n-1\\
(-1)^n qx, & k=n .
\end{cases}
\end{equation}

Denote by $\Td_{\sigma}^N$ the degree $N$  homogeneous
part of $\Td_{\sigma}$. 
Then from \eqref{S_sigma} 
$\Td_\sigma^{N}$ is given as follows:
\begin{equation}\label{ToddN}
\begin{split}
\Td_{\sigma}^{N}
=&qxy\sum_{k=-1}^{n-1}(-1)^{k+1}\sum_{i=0}^{N-2}(-1)^N\frac{B_{i+1}}{(i+1)!}\frac{B_{N-i-1}}{(N-i-1)!}M_k^{N-2-i}M_{k+1}^{i}\\
&+(-1)^Nqxy\frac{B_{N}}{N!}\sum_{k=-1}^{n-1}(-1)^{k+1}\frac{M_k^{N}+M_{k+1}^{N}}{M_kM_{k+1}}.
\end{split}
\end{equation}

Since 
for $k\ge 0$ we have
$$
M_{k-1}-M_{k+1}=a_{k+1}M_k,
$$
\begin{equation}\label{qxy}
\begin{split}
&qxy\sum_{k=-1}^{n-1}(-1)^{k+1}\frac{M_k^{N}+M_{k+1}^{N}}{M_kM_{k+1}}\\
&=qxy\left(\sum_{k=0}^{n-1}(-1)^ka_{k+1}\sum_{i=0}^{N-2}M_{k-1}^{N-2-i}M_{k+1}^i\right)+M_0^{N-1}x+M_{n-1}^{N-1}y.
\end{split}
\end{equation}

Therefore, plugging \eqref{qxy} into \eqref{ToddN},  we obtain
\begin{equation}\label{eq1}
\begin{split}
&\Td_{pq}^{N}
=qxy\sum_{k=-1}^{n-1}(-1)^{k+1}\sum_{i=0}^{N-2}(-1)^{N}\frac{B_{i+1}}{(i+1)!}\frac{B_{N-i-1}}{(N-i-1)!}M_k^{N-2-i}M_{k+1}^{i}\\
&+(-1)^Nq\frac{B_{N}}{N!}xy\left(\sum_{k=0}^{n-1}(-1)^ka_{k+1}\sum_{i=0}^{N-2}M_{k-1}^{N-i-2}M_{k+1}^i
\right)+\frac{B_{N}}{N!}M_0^{N-1}x+\frac{B_{N}}{N!}M_{n-1}^{N-1}y.
\end{split}
\end{equation}


Consequently, we obtain the following integrality involving $\Td_{pq}^{N}$.
\begin{thm}\label{Todd_N_part}
Suppose $N$ is an even positive integer.
Let $\frac{\alpha_k}{\beta_k}$ be the reduced fraction of $B_k\neq 0$ with $\beta_k>0$ 
and 
$$
r_N:=\text{L.C.M.}\left\{\text{Denominator of} \,\,\beta_{N}\begin{pmatrix}N \\i+1\end{pmatrix} B_{i+1}B_{N-i-1}\Big| \text{$i$ odd}, 0\leq i \leq {N-2}\right\}.
$$
Then  we have
\begin{equation*}
\begin{split}
&N! \beta_{N}r_N\Td_{pq}^{N}(x,y)\\
&-\alpha_{N}r_N(x+py)^{N-1}x-\alpha_{N}r_N\left((-1)^{n-1}q_{n-1}x+y\right)^{N-1}y\in q\FZ[x,y].
\end{split}
\end{equation*}
\end{thm}
\begin{proof}
By multiplying  $N!\beta_{N}$ on the equation (\ref{eq1}), we have that 
\begin{equation}\label{eq2}
\begin{split}
&N!\beta_{N}\Td_{pq}^{N} - \alpha_{N}M_0^{N-1}x- \alpha_{N}M_{n-1}^{N-1}y\\
=&qxy\sum_{k=-1}^{n-1}(-1)^{k+1}\sum_{i=0}^{N-2}(-1)^{N}\beta_{N}
\begin{pmatrix}N \\i+1\end{pmatrix} B_{i+1}B_{N-i-1}M_k^{N-i-2}M_{k+1}^{i}\\
&+(-1)^Nq\alpha_{N}xy\left(\sum_{k=0}^{n-1}(-1)^ka_{k+1}\sum_{i=0}^{N-2}M_{k-1}^{N-2-i}M_{k+1}^i
\right).
\end{split}
\end{equation}

Since $M_i\in \FZ[x,y]$ for every $i$,  multiplying \eqref{eq2} by $r_N$, we conclude the proof. 
\end{proof}

\begin{proof}[Proof of Thm.\ref{general}]
If we read coefficient of $x^{i}y^{N-i}$ in the previous theorem, 
we obtain the formula of 
Thm.\ref{general}.
\end{proof}
\begin{remark}
In Thm.\ref{general}, if we consider the case of $i=1=j$, so $N=2$, then 
the denominator $\beta_2$ of $B_2$ is $6$ and $r_N$ is $1$. Thus, we have $R(1,1)=12$.
This is the  case of Thm. \ref{Rademacher} considered by Rademacher.
\end{remark}

Accoding to Thm.\ref{general}, we associate generalized Kloosterman sums to the generalized Dedekind sums
as follows:
\begin{equation}\label{gene}
\begin{split}
\sum_{\substack{0<p<q\\(p,q)=1}}&\bbe\left(R(i,j)q^{N-2}s_{ij}(p,q)\right)\\
&=\sum_{\substack{0<p<q\\ pp'\equiv 1\pmod{q}}} \bbe\left( \frac{(p')^{i}\alpha_{N}r_N\begin{pmatrix}N-1 \\i\end{pmatrix} +p^{j}\alpha_{N}r_N\begin{pmatrix}N-1 \\j\end{pmatrix}}{q}\right)\\
&=K_{ij}\left(\alpha_{N}r_N\begin{pmatrix}N-1 \\i\end{pmatrix},
\alpha_{N}r_N \begin{pmatrix}N-1 \\j\end{pmatrix},q\right). 
\end{split}
\end{equation}

\section{Bounds for Generalized Kloosterman sums}

In this section, we are going to investigate the Weil type bound for the generalized Kloosterman sums $K_{ij}(k,\ell,q)$. 
This estimate amounts basically to asking the weight of the cohomology of 
a certain $\ell$-adic sheaf.
From the bound we will show the Weyl's equidistribution criterion for 
$\left\{\left<R(i,j)q^{N}s_{ij}(p,q)\right>\right\}$ is fulfilled for $i+j$ even. 
This will conclude the proof of the equidistribution of the generalized Dedekind sums.

\subsection{Weight of $\ell$-adic sheaf and exponential sums}
Let us first recall the work of Denef-Loeser(\cite{De-Loe}).
Let $X$ be a scheme of finite type over $k:=\FF_q$ and $\psi : k \to \FC^*$ be a nontrivial additive 
character. Then a $\overline{\FQ}_\ell$-sheaf $\mathcal{L}_\psi$  on $\FA^1_k$ is associated to
$\psi$ and the Artin-Schreier covering $t^q-t=x$.
For a morphism $f: X \to \FA^1_k$, the exponential sum 
$$
S(f) = \sum_{x\in X(k)} \psi \left(f(x)\right)
$$
is defined. Let $Fr$ denote the (geometric) Frobenius action.
Grothendieck's trace formula 
identifies this exponential sum with the trace of the Frobenious action on the cohomology:
$$
S(f)=\sum_i (-1)^i \tr\left(Fr^*|H^i_c (X\otimes \bar{k}, f^*\mathcal{L}_\psi)\right).
$$

For $X=T^n_k$, if a map $f:X \to A^1_k$ is given by a Laurent polynomial 
$f=\sum_{i\in\FZ^n} c_i x^i$, the Newton polyhedron $\Delta_\infty (f)$ is defined as
the convex hull of $\{i\in \FZ^n| c_i \ne 0 \}$ in $\FR^n$.
$f$ is said to be \textit{non degenerate w.r.t. $\Delta_\infty(f)$} 
if for every face $\sigma$ of $\Delta_\infty(f)$ that does not contain $0$,  
the locus 
$$
\frac{\partial f_\sigma}{\partial x_1} = \cdots = \frac{ \partial f_\sigma}{\partial x_n}=0
$$ 
is empty.
Then a result of Denef and Loeser(Thm.1.3. in \cite{De-Loe}) is stated as follows:
\begin{thm}[Denef-Loeser\cite{De-Loe}]
\label{De-Loe}
Suppose $f:T^n_k \to \FA^1_k$ is nondegenerate w.r.t. $\Delta_\infty(f)$ and $\dim \Delta_\infty(f)=n$. Then we have
\begin{enumerate}
\item $H^i_c(T^n_{\bar{k}}, f^*\sL_\psi)=0$ for $i\ne n$,
\item $\dim H^n_c(T^n_{\bar{k}}, f^*\sL_\psi) = n! \Vol(\Delta_\infty(f))$.\\
If moreover the interior of $\Delta_\infty(f)$ contains $0$, then
\item $H^n_c(T^n_k, f^*\sL_\psi)$ is pure of weight $n$ (ie. all Frobenius eigenvalues have
absolute value $q^{n/2}$.
\end{enumerate}
\end{thm} 

If $f$ satisfies the conditions of the above theorem, the trace formula is simplified as
$$
S(f)= (-1)^n \tr\left(Fr^*|H^n_c (X\otimes \bar{k}, f^*\mathcal{L}_\psi)\right).
$$

Then the Weil type bound is a simple consequence of the purity result:
$$
\left|S(f)\right| \le \sum \left|\text{Frobenius eigenvalue}\right|    \le C_f {q^{n/2}},
$$
where $C_f =  \dim H^n_c(T^n_{\bar{k}}, f^*\sL_\psi) = n! \Vol(\Delta_\infty(f)) $.

By the fundamental result of Deligne in \cite{Deligne}, we know  that $H^n_c(T^n_{\bar{k}},f^*\sL_\psi)$ has mixed weight
$\le n$. This is already enough to obtain the Weil bound, but to obtain the dimension, we need the theorem of Denef-Loeser.

First, for $q=p$ and $f(z) = k z^i + \ell z^{-j}$ non degenerate, we obtain the Weil bound for the generalized Kloosterman sum.
Second,  we reduce the generalized Kloosterman sum of composite modulus to a product of those of $p$-primary modulus.  
Besides, we will consider those exceptional cases separately. 
Altogether, this bound yields the Weyl's criterion for equidistribution.

\subsection{Reduction to non degenerate case}

Let us write first the Weil bound for the generalized Kloosterman sum of prime modulus 
in non degenerate case.

\begin{lem}[Nondegenerate case]
\label{nondegenerate}
 Let $p$ be a prime and $p$ does not divide $i$ and $j$.
Suppose that $k$ and $\ell$ are not divisible by $p$.
Then
we have 
$$
|K_{ij}(k,\ell,p) |\le (i+j) p^{1/2}
$$
\end{lem}
\begin{proof}
The condition on $i,j, k,\ell$ ensures that $f(z)= kz^i +\ell z^{-j}$ nondegenerate w.r.t. its Newton polyhedron. 
It is a direct consequence of Thm. \ref{De-Loe} due to Denef-Loeser.
\end{proof}

Suppose that $f$ fails to be non degenerate for a given prime $p$. 
This happens when $p$ divides at least one of $i$, $j$, $k$ or $\ell$. 
\textit{A priori} $f$ can be degenerate for only finitely many cases of prime $p$. 
Since we vary $p$ for fixed $i,j$, it is easy to see
$$
|K_{ij}(k,\ell,p)| \le C p^{1/2}
$$
for a constant $C$ independent of $p$ but determined by $i$ and $j$.

If $q$ has many prime factors, we need to reduce the case to the non-degenerate. 
This will be justified through the next  two lemmas.

First we consider the case $q$ being power of a prime $p$. 
When either $k$ or $\ell $ is divisible by some power of $p$, the first reduction is as follows:
\begin{lem} \label{lemma1}
Let $p$ be a fixed prime and $k= k' p^\beta$, $\ell=\ell' p^\beta$ for $p^\beta || gcd(k,\ell)$.
Then for given  $i$ and $j$,
we have 
$$
K_{ij}(k,\ell,p^{\alpha})=p^{\beta}K_{ij}(k',\ell',p^{\alpha-\beta}).
$$
\end{lem}
\begin{proof}
We note that  an element $z\in \left(\FZ/p^{\alpha}\FZ\right)^*$ is expressed as $$z=p^{\alpha-\beta}x+y$$ for $x\in \FZ/p^{\beta}\FZ$ and $y\in \left(\FZ/p^{\alpha-\beta}\FZ\right)^*$.

Thus, we  find that
\begin{equation}
\begin{split}
K_{ij}(k,\ell,p^{\alpha})&=\sum_{z\in (\FZ/p^{\alpha}\FZ)^*}\bbe\left(\frac{kz^i+\ell z^{-j}}{p^{\alpha}}\right)
=\sum_{x\in \FZ/p^{\beta}\FZ}\sum_{y\in (\FZ/p^{\alpha-\beta}\FZ)^*}\bbe\left(\frac{k'y^i+\ell'y^{-j}}{p^{\alpha-\beta}}\right)\\
&=p^{\beta}K_{ij}(k',\ell',p^{\alpha-\beta}).
\end{split}
\end{equation}
\end{proof}
After the previous lemma, we can pull out $p$-factors out of $k, \ell$.
For non degenerate $k$ and $\ell$, we obtain the following bound:
\begin{lem}\label{lemma2} 
Let $p$ be a prime and suppose at least one of  $k$ or $\ell$ is indivisible by $p$. 
Then, for given 
positive integers $i,j$  
$$
\left|K_{ij}(k,\ell,p^{\alpha})\right|\leq i j (i+j)^{\frac{3}{2}}p^{\frac{\alpha}{2}}.
$$
\end{lem}
\begin{proof}
Applying  Lemma 12.2-3 of \cite{Iwaniec} for the cases $\alpha$ even and odd, 
we can reduce the estimation to that of simpler sums respectively.

For $\alpha=2\beta$,
we have
\begin{equation}\label{palpha}
\begin{split}
\left|K_{ij}(k,\ell ,p^{2\beta})\right|&=
p^{\beta} \left|\sum_{\substack{x\in (\FZ/p^\beta \FZ)^*\\
 i k x^{i+j}=j \ell }}
 \bbe\left(\frac{k x^i+\ell x^{-j}}{p^{2\beta}}\right)\right| \\
  &\le p^{\beta} \sum_{\substack{x\in (\FZ/p^\beta \FZ)^* \\ i k x^{i+j}=j \ell }}1 \le p^\beta ij(i+j).
\end{split}
\end{equation}

For $\alpha= 2\beta+1$, 
\begin{equation}\label{palpha1}
K_{ij}(k,\ell,p^{2\beta+1})=p^\beta \sum_{\substack{x\in (\FZ/p^{\beta}\FZ)^*\\ i k x^{i+j}=j \ell }}\bbe\left(\frac{k x^i+\ell x^{-j}}
{p^{2\beta+1}}\right)G_p(x)
\end{equation}
where
$$
G_p(x)=\sum_{y\in \FZ/p\FZ}e_p\left(d(x)y^2+h(x)p^{-\beta}y\right).
$$
Here, 
\begin{equation*}
\begin{split}
&e_p(x)=\bbe\left(\frac{x}{p}\right),\\
&d(x)=\frac{1}{2}\left(k i \left(i-1\right)x^{i-2}+\ell j \left(j+1\right)x^{-j-2}\right) \\
\text{and}\quad& h(x)=k i x^{i-1}-\ell j x^{-j-1}.
\end{split}
\end{equation*}

Note that
$$
G_p(x)=
\begin{cases}
p & \text{for $2d(x)\equiv 0,\,\,\, h(x)p^{-\beta}\equiv 0 \pmod{p}$,}\\
0 & \text{for $2d(x)\equiv 0 , \,\,\,h(x)p^{-\beta}\not\equiv 0 \pmod{p}$.} 
\end{cases}
$$
and  $\left|G_p(x)\right|\leq \sqrt{p}$ for $2d(x)\not\equiv 0 \pmod{p}$.
If 
$2d(x)\equiv h(x)p^{-\beta}\equiv 0 \pmod{p},$ then 
$p$ divides $  j (i+j)$, thus $p\le i+j$.
Thus we obtain a general bound for $G_p(x)$:
$$
\left|G_p(x)\right|\leq p^{\frac{1}{2}}(i+j)^{\frac{1}{2}}.
$$
This yields the desired bound for odd $\alpha$. 
\end{proof}

These two lemmas imply the Weil bound for $q=p^\alpha$ as follows: 
 \begin{prop}\label{pro1}
For all positive integers $i,j$ and positive prime $p$, $q=p^{\alpha}$, 
we have 
$$
\left|K_{ij}(k,\ell,p^{\alpha})\right|\leq ij(i+j)^{\frac{3}{2}}(k,\ell,p^{\alpha})^{\frac{1}{2}}p^{\frac{\alpha}{2}}.
$$
\end{prop}

Now we need  reduction to a single prime factor when there are several prime factors of $q$.

\begin{lem}\label{pro2} 
Let $q_1>1$ and $q_2>1$ be relatively prime integers.
For $k,\ell$, let $k_i$ and $\ell_i$ be the mod $q_i$ Chinese remainder(ie. Under the isomorphism
$\FZ/q_1q_2\FZ \to \FZ/q_1 \times \FZ/q_2\FZ$, $k\mapsto (k_1,k_2)$ and $\ell\mapsto (\ell_1,\ell_2)$ ).
Then we have
$$
K_{ij}(k,\ell,q_1q_2)=K_{ij}(k_1,\ell_1,q_1)K_{ij}(k_2,\ell_2,q_2).
$$
\end{lem}
\begin{proof}
This is an easy consequence of Fubini theorem.
\end{proof}
 
Since $\sum_{q<x} \phi(q)\sim x^2$ for Euler-phi function $\phi$, we come to 
the proof of the main theorem 
from the Weyl's criterion for equidistribution and forthcoming Prop.\ref{bound}.

Combining Prop.\ref{pro1} and Lemma \ref{pro2},  
we  have  the following result:
\begin{prop}\label{bound}
For all positive integers $i,j$ and $q$,
$$
\left|K_{ij}(k,\ell,q)\right| \leq \left(ij(i+j)^{\frac{3}{2}}\sqrt{(k,\ell)}\right)^{\omega(q)} \sqrt{q},$$
where $\omega(q)$ is the number of prime factors of $q$.
\end{prop}
\begin{proof}
Let $q=p_1^{n_1}p_2^{n_2}.....p_{\omega(q)}^{n_{\omega(q)}}$, for distinct primes $p_1,p_2, \cdots ,p_{\omega(q)}$. 

Since  
$(k,\ell,p^{\alpha})\leq(k,\ell),$ Prop \ref{pro1} directly implies that  for any prime $p$, 
$$
\left|K_{i,j}(k,\ell ,p^{\alpha})\right|\leq ij(i+j)^{\frac{3}{2}}\sqrt{(k,\ell)}\sqrt{p^{\alpha}}.
$$
Finally, after the multiplicativity of Kloosterman sum as in  Lem \ref{pro2}, the 
corollary is proved. 
\end{proof}

\subsection{Proof of Thm. \ref{main}}
Finally, we deduce the Weyl's criterion from the bound of the generalized Kloosterman sums.

$\omega(q)$  in Prop.\ref{bound} has well-known estimate:
\begin{equation}
\omega(q)\sim \log\log q.
\end{equation}

For sufficiently large $q$,
$$
\left(ij(i+j)\sqrt{(k,\ell)}\right)^{\omega(q)}\leq \left(ij(i+j)^{\frac{3}{2}}\sqrt{(k,\ell)}\right)^{c \log \log q}\leq (\log q)^{c\log(i+j)^{\frac{3}{2}}ij\sqrt{(k,\ell)}}.
$$
Thus, we have that  for any $\epsilon>0$,
$$
\left(ij(i+j)^{\frac{3}{2}}\sqrt{(k,\ell)}\right)^{\omega(q)}<<q^{\epsilon}.
$$

Therefore, we have the following Weil type bound:
\begin{thm}
For given pair of positive integers $i,j$,
$$
\left|K_{ij}(k,\ell,q)\right|<< q^{\frac{1}{2}+\epsilon}, \,\,\, \forall \epsilon>0.
$$
\end{thm}

Now, 
we show the Weyl's criterion of Generalized dedekind sum from the
following estimate:
\begin{equation}
\begin{split}
\sum_{0<q<x}&\sum_{\substack{0<p<q\\(p,q)=1}} \bbe\left(m R\left(i,j\right) q^{N-2} s_{i,j}\left(p,q\right)\right)\\
&= \sum_{0<q<x}K_{i,j}\left(m\alpha_{N}r_N\begin{pmatrix}N-1 \\i\end{pmatrix},
m\alpha_{N}r_N \begin{pmatrix}N-1 \\j\end{pmatrix},q\right)\leq x^{\frac{3}{2}+\epsilon}. 
\end{split}
\end{equation}

Consequently,  
Weyl's equidistribution criterion is fulfilled for
the fractional part of 
$R(i,j) q^{N-2} s_{i,j}(p,q)$:
\begin{equation}
\begin{split}
&E_{ij}(m,x)= \\
& \frac{1}{\# \left\{(p,q)|\gcd(p,q)=1, p<q\le x \right\} } \sum_{0<q<x}\sum_{\substack{0<p<q\\(p,q)=1}} \bbe\left(m R\left(i,j\right) q^{N-2} s_{i,j}\left(p,q\right)\right)\rightarrow0,
\end{split}
\end{equation}
as $x\rightarrow \infty$.
Therefore the proof of the main theorem is finished.

\end{document}